\documentclass[11pt,letterpaper,reqno]{amsart}
\usepackage{tikz}
\usetikzlibrary{positioning, shapes.geometric, arrows.meta}
\usepackage{amssymb}
\usepackage{amsmath}
\usepackage{amsthm}
\usepackage{amsfonts}
\usepackage{bbm}
\usepackage{enumitem} 
\usepackage{pgfplots}
\pgfplotsset{compat=1.18} 
\usepackage{booktabs}

\usepackage{graphicx}
\usepackage[T1]{fontenc}
\usepackage{doi}
\usepackage{float} 
\usepackage{bookmark}
\usepackage{hyperref}
\hypersetup{pdfstartview={FitH}}

\numberwithin{equation}{section}

\newtheorem{thm}{Theorem}[section]
\newtheorem{prop}[thm]{Proposition}
\newtheorem{lem}[thm]{Lemma}
\newtheorem{prob}[thm]{Problem}
\newtheorem{cor}[thm]{Corollary}
\theoremstyle{definition}
\newtheorem{defin}[thm]{Definition}
\theoremstyle{remark}
\newtheorem{remark}[thm]{Remark}
\newtheorem{example}[thm]{Example}

\DeclareMathOperator{\Tr}{Tr}

\DeclareMathOperator{\Dom}{Dom}

\newcommand{\A}{\mathcal{A}}
\newcommand{\M}{\mathcal{M}}
\newcommand{\RR}{\mathbb{R}}
\newcommand{\CC}{\mathbb{C}}

\newcommand{\Id}{\mathbf{1}}
\newcommand{\ind}{\mathbbm{1}} % indicator / spectral projection symbol
\newcommand{\dd}{\,\mathrm{d}}

\newcommand{\Apos}{\A_{+}}
\newcommand{\Ainvpos}{\A_{++}}
\newcommand{\Mpos}{\M_{+}}
\newcommand{\Minvpos}{\M_{++}}

\newcommand{\nablaop}{\mathbin{\nabla}}
\newcommand{\sharpop}{\mathbin{\#}}

\newcommand{\etaop}{\eta}

\newcommand{\normtau}[1]{\left\|#1\right\|_{2,\tau}}

\title[Metric Properties: From $S$-Divergence to Quantum Jensen Divergence]
{Metric Properties: From $S$-Divergence to Quantum Jensen Divergence}

\author{Teng Zhang}
\address{School of Mathematics and Statistics, Xi'an Jiaotong University, Xi'an, P.R.China 710049}
\email{teng.zhang@stu.xjtu.edu.cn}


\begin{document}
	
\begin{abstract}
	We extend the trace-logarithmic $S$-divergence from matrices to tracial $C^*$-algebras and finite von Neumann algebras, and show that its square root defines a metric on the invertible positive cone. We also prove an integral representation of the quantum Jensen--Shannon divergence in terms of shifted trace-log distances, implying metricity of its square root on the full positive cone in the same tracial framework. In the matrix case, we answer two questions of Virosztek \cite{Vir21} on Hilbertianity. Finally, we show that symmetric quantum Jensen divergences generated by non-affine operator convex functions yield metrics in the tracial setting via a Nevanlinna--Stieltjes type representation of the derivative, which generalizes a result of Carlen, Lieb and Seiringer.
\end{abstract}

	\maketitle
%	\tableofcontents
	
	% ============================================================
	\section{Introduction}
	% ============================================================
	
	Throughout this paper, let $M_n^+(\CC)$ denote the cone of all $n\times n$ positive semidefinite matrices,
	and let $M_n^{++}(\CC)$ denote the cone of all $n\times n$ positive definite matrices.
	
	Let $(\A,\tau)$ be a unital $C^*$-algebra equipped with a faithful tracial state
	\[
	\tau(\Id)=1,\qquad \tau(x^*x)=0 \ \Rightarrow\ x=0,\qquad \text{and}\qquad
	\tau(xy)=\tau(yx)\ \ \text{for all }x,y\in\A.
	\]
We also use the associated noncommutative $L^2$-norm
$\|x\|_{2,\tau}:=\tau(x^*x)^{1/2},$
whenever $\tau$ is understood from context.
	We write $\Apos$ for the positive cone of $\A$, and $\Ainvpos$ for its invertible positive cone. We also write
	\[
	\mathcal S_\tau:=\{\rho\in\A_{+}:\tau(\rho)=1\},\qquad
	\mathcal S_{\tau,++}:=\mathcal S_\tau\cap \A_{++}.
	\]

	Let $(\M,\tau)$ be a finite von Neumann algebra with a faithful normal tracial state $\tau$,
	and denote by $\M^\times$ the group of invertible elements of $\M$.
	Write $\M_{\mathrm{sa}}=\{x\in\M:\ x^*=x\}$, and denote by $\Mpos$ and $\Minvpos$
	the positive cone and the invertible positive cone of $\M$, respectively.
	
	\subsection{The $S$-divergence}
	
The $S$-divergence was introduced by Sra \cite{Sra16} in the study of the open convex cone
$M_n^{++}(\CC)$.
In contrast, the affine-invariant Riemannian distance
\[
\delta_R(A,B)=\bigl\|\log\bigl(B^{-1/2}AB^{-1/2}\bigr)\bigr\|_F,
\qquad A,B\in M_n^{++}(\CC),
\]
where $\|\cdot\|_F$ denotes the Frobenius norm, arises from the nonpositively curved
Riemannian geometry of $M_n^{++}(\CC)$ (see, e.g., \cite[Ch.~10]{BH99} and \cite[Ch.~6]{Bha07}).
However, outside this specific geometric framework, there is in general no canonical
“natural” choice of distance on $M_n^{++}(\CC)$.

	Motivated by ideas from convex optimization and information geometry \cite{BV04,CSBP11},
	Sra proposed the symmetric trace-logarithmic expression
	\begin{equation}\label{eq:sra}
		\delta_S^2(A,B)
		=\log\det\!\Bigl(\frac{A+B}{2}\Bigr)-\frac12\log\det A-\frac12\log\det B,
		\qquad A,B\in M_n^{++}(\CC),
	\end{equation}
	also known as the Jensen--Bregman LogDet divergence (or symmetric Stein divergence);
	see \cite{CCR08,CM12,Sra16}. A key feature proved in \cite{Sra16} is that
	$\delta_S:=\sqrt{\delta_S^2}$ satisfies the triangle inequality on $M_n^{++}(\CC)$, hence defines a genuine metric.
	
	\begin{thm}[Sra]\label{thm:Sra}
		Let $A,B\in M_n^{++}(\CC)$. Define $\delta_S$ by \eqref{eq:sra}. Then $\delta_S$ is a metric on $M_n^{++}(\CC)$.
		Moreover, it does not admit any Hilbert space embedding for any $n\ge 2$.
	\end{thm}
	
	Besides its intrinsic geometric interest, $\delta_S$ is computationally attractive and has been used effectively
	in applications such as similarity search for covariance descriptors \cite{CSBP11}, which further motivates the study
	of trace-logarithmic metrics beyond finite dimensions.
	
	There is a substantial literature extending log-determinant divergences to infinite-dimensional operator settings.
	For example, Minh \cite{Min17,Min20} developed infinite-dimensional LogDet divergences for unitized trace-class and Hilbert--Schmidt perturbations
	via extended Fredholm-type determinants. From the viewpoint of finite von Neumann algebras, the
	Fuglede--Kadison determinant $\Delta$ \cite{FK52} provides a canonical replacement for $\det$, and ``distance-like'' functionals
	built from $\Delta$ and functional calculus have been considered in connection with isometry problems (see, e.g.,
	Ga\'al--Nagy--Szokol \cite{GNS19}).
	
	In the matrix case, $\log\det X=\Tr(\log X)$, suggesting a natural extension to tracial $C^*$-algebras and finite von Neumann algebras
	by replacing $\Tr$ with a faithful tracial state.
	This raises the question of whether the $S$-divergence remains a metric in the $C^*$-algebraic setting.
	In this paper, we answer this in the affirmative.
	
	\begin{thm}\label{thm:main-tracelog-Cstar}
		Let $(\A,\tau)$ be a unital $C^*$-algebra with a faithful tracial state. For $A,B\in\A_{++}$ define
		\begin{equation}\label{eq:d_tau}
			d^2_\tau(A,B)
			:=\tau\!\left(\log\!\left(\frac{A+B}{2}\right)\right)-\frac12\tau(\log A)-\frac12\tau(\log B).
		\end{equation}
		Then $d_\tau(A,B)$ is a metric on $\A_{++}$.
	In particular, its restriction to $\mathcal S_{\tau,++}$ is a metric.
	\end{thm}
	We call $d_\tau$ the \emph{trace-log distance} (or \emph{$S$-divergence metric}) associated with $(\A,\tau)$.
	
	\subsection{Quantum Jensen--Shannon divergence}
	
	Let $\eta(x)=x\log x$ for $x>0$ and $\eta(0)=0$ (continuous extension).
	The quantum Jensen--Shannon divergence (QJSD) is a quantum analogue of the classical Jensen--Shannon divergence,
	systematically introduced by Majtey--Lamberti--Prato \cite{MLP05}. A standard matrix form is
	\begin{equation}\label{eq:qjsd-def-matrix}
		J_{\eta}(A,B)
		:=\frac12\Tr(\eta(A))+\frac12\Tr(\eta(B))-\Tr\!\left(\eta\!\left(\frac{A+B}{2}\right)\right),
		\qquad A,B\in M_n^{+}(\CC).
	\end{equation}
	The QJSD is nonnegative, symmetric, and bounded; compared with the quantum relative entropy it is often regarded as smoother in applications.
	It has found broad applications beyond quantum theory~\cite{DLH11,RPJB16,RFZ10},
	including complex network theory~\cite{DNAL15,DPA15}, pattern recognition~\cite{BRTH15},
	graph theory~\cite{RTHW13}, and chemical physics~\cite{AAL09}.
	
Virosztek \cite{Vir21} proved that the square root of the quantum Jensen--Shannon divergence defines a genuine metric on the full cone of positive matrices, in every dimension.

\begin{thm}[Virosztek]\label{thm:Virosztek}
	Let $A,B\in M_n^{+}(\CC)$. Define $J_\eta(A,B)$ by \eqref{eq:qjsd-def-matrix}.
	Then $\sqrt{J_{\eta}(A,B)}$ is a metric on $M_n^{+}(\CC)$.
\end{thm}

Having established metricity, Virosztek  \cite{Vir21} further asked whether this metric is induced by a Hilbert space embedding.

\begin{prob}[Virosztek]\label{prob:Virosztek_1}
	Is the metric induced by the quantum Jensen--Shannon divergence Hilbertian on $M_n^+(\CC)$ for $n\ge 2$?
\end{prob}

A related question \cite{Vir21} concerns the physically important subset of density matrices.

\begin{prob}[Virosztek]\label{prob:Virosztek_2}
	Is the QJSD metric Hilbertian on the space of $n\times n$ density matrices for $n\ge 3$?
\end{prob}

Our next theorem answers both questions in the negative.

\begin{thm}\label{thm:main-not-hilbertian}
	Let $\sqrt{J_{\eta}}$ be the QJSD metric on $M_n^{+}(\CC)$ defined by \eqref{eq:qjsd-def-matrix}.
	\begin{enumerate}[label=\textup{(\arabic*)}, leftmargin=2em]
		\item If $n\ge 2$, then $\sqrt{J_{\eta}}$ is not Hilbertian on $M_n^{++}(\CC)$.
		Consequently, it is not Hilbertian on $M_n^{+}(\CC)$.
		\item If $n\ge 3$, then $\sqrt{J_{\eta}}$ is not Hilbertian on the density matrices
	$
		\mathcal D_n:=\{\rho\in M_n^{+}(\CC):\Tr(\rho)=1\}.
	$
	\end{enumerate}
\end{thm}

In addition to these finite-dimensional results, we also show that Virosztek's metricity theorem extends to the setting of tracial $C^*$-algebras.

\begin{thm}\label{thm:qjsd-Cstar}
	Let $(\A,\tau)$ be a unital $C^*$-algebra with a faithful tracial state, and let $\eta(x)=x\log x$ with $\eta(0)=0$.
	For $A,B\in\A_{+}$ define
	\[
	J_{\tau,\eta}(A,B):=\frac12\tau(\eta(A))+\frac12\tau(\eta(B))-\tau\!\left(\eta\!\left(\frac{A+B}{2}\right)\right).
	\]
	Then the square root $\sqrt{J_{\tau,\eta}(A,B)}$ is a metric on $\A_{+}$.
	In particular, its restriction to the $\tau$-state space $\mathcal S_\tau$ is a metric.
\end{thm}

The QJSD fits into a broader family of Jensen-type divergences generated by convex (and, in the noncommutative setting, operator convex) functions. This viewpoint will be useful both conceptually and technically in later sections.

\subsection{Quantum Jensen $f$-divergences and trace Jensen gaps}

For a continuous function \(f:[0,\infty)\to\mathbb{R}\), the associated \emph{trace Jensen gap} on \(M_n^{+}(\mathbb{C})\) is defined by
\begin{equation*}
	J_f(A,B)
	:=\frac12\,\mathrm{Tr}\bigl(f(A)\bigr)+\frac12\,\mathrm{Tr}\bigl(f(B)\bigr)
	-\mathrm{Tr}\!\left(f\!\left(\frac{A+B}{2}\right)\right),
	\qquad A,B\in M_n^{+}(\mathbb{C}),
\end{equation*}
whenever the terms are well-defined (this is automatic under bounded functional calculus).
To ensure that \(J_f\ge 0\), we additionally assume that \(f\) is operator convex and not affine (otherwise, \(J_f\equiv 0\)). 

Let $\mathcal S(\CC^2)$ denote the space of $2\times 2$ density matrices.
On the qubit state space $\mathcal S(\CC^2)$, the following strengthening was proved by Carlen--Lieb--Seiringer;
see \cite[Theorem~2]{Vir21}.

\begin{thm}[Carlen--Lieb--Seiringer]
	For an operator convex function \(f:[0,\infty)\to\mathbb{R}\), the square root of the symmetric quantum Jensen \(f\)-divergence defined by
	\[
	J_f(\rho,\sigma)
	:=\frac12\bigl(\Tr f(\rho)\bigr)+\frac12\bigl(\Tr f(\sigma)\bigr)
	-\Tr f\!\left(\frac{\rho+\sigma}{2}\right),
	\qquad (\rho,\sigma\in \mathcal S(\mathbb{C}^2)),
	\]
	is a true metric on \(\mathcal S(\mathbb{C}^2)\); moreover, it admits a Hilbert space embedding.
\end{thm}

Beyond the qubit setting, however, such an embedding phenomenon cannot hold in full generality, see \cite{Vir21}.
Indeed, the cone of \(2\times 2\) positive definite matrices equipped with the Jensen divergence corresponding to the operator convex function \(x\mapsto -\log x\) is known not to admit any Hilbert space embedding (Theorem~\ref{thm:Sra}).
This indicates that the Carlen--Lieb--Seiringer theorem is optimal: one cannot, in general, go beyond the qubit state space when seeking Hilbertianity simultaneously for all operator convex generators.
(To be precise, \(x\mapsto -\log x\) is operator convex only on \((0,\infty)\) and not on \([0,\infty)\), but the integral representation
\[
\log x=\int_{0}^{\infty}\left(\frac{1}{1+t}-\frac{1}{x+t}\right)\,dt
\]
shows that, from the viewpoint of negative definite kernels, it behaves similarly to operator convex functions on \([0,\infty)\).) 

An important and challenging problem is whether the square root of the symmetric quantum Jensen $f$-divergence defines a metric on $M_n(\mathbb{C})^{+}$. For $f(x)=x\log x$, it reduces to Theorem~\ref{thm:Virosztek}.
We next provide a general mechanism showing that operator convex Jensen generators still yield metrics (without claiming Hilbertianity in general) in the tracial $C^*$-algebra setting.

\begin{thm}\label{thm:main-opconvex}
	Let $(\A,\tau)$ be a unital $C^*$-algebra with a faithful tracial state.
	Let $f:(0,\infty)\to\RR$ be an operator convex function which is not affine, and assume that $f$ extends continuously to $[0,\infty)$ with finite $f(0)$.
	For $A,B\in\A_{+}$ define
	\[
	J_{\tau,f}(A,B)
	:=\frac12\tau\bigl(f(A)\bigr)+\frac12\tau\bigl(f(B)\bigr)
	-\tau\!\left(f\!\left(\frac{A+B}{2}\right)\right).
	\]
Then the square root $\sqrt{J_{\tau,f}}$ defines a metric on $\A_{+}$.
	More precisely, there exist $b\ge 0$ and a positive Borel measure $\nu$ on $(0,\infty)$ satisfying
$
	\int_{(0,\infty)}\frac{1}{1+t}\,\dd\nu(t)<\infty
$
	such that for all $A,B\in\A_{+}$,
	\begin{equation}\label{eq:opconvex-decomp-intro}
		J_{\tau,f}(A,B)
		=\frac{b}{8}\,\normtau{A-B}^{2}
		+\int_{(0,\infty)} t\,d_\tau(A+t\Id,B+t\Id)^2\,\dd\nu(t).
	\end{equation}
\end{thm}
As an immediate corollary of Theorem~\ref{thm:main-opconvex}, we have
\begin{cor}
Let $f:(0,\infty)\to\RR$ be a  operator convex function which is not affine, and assume that $f$ extends continuously to $[0,\infty)$ with finite $f(0)$.
Then the square root of the symmetric quantum Jensen $f$-divergence defined by
	\[
	J_f(A,B)
	:=\frac12\bigl(\Tr f(A)\bigr)+\frac12\bigl(\Tr f(B)\bigr)
	-\Tr f\!\left(\frac{A+B}{2}\right),
	\qquad (A,B\in  M_n^{+}(\mathbb{C})),
	\]
	is a true metric on \(M_n^{+}(\mathbb{C})\).
\end{cor}
\medskip
	\noindent\textbf{Organization of this paper.}
	In Section~\ref{sec:prelim}, we recall the Fuglede--Kadison determinant and generalized singular value functions.
	In Sections~\ref{sec:scalar}--\ref{sec:triangle} we prove the triangle inequality for $d_\tau$ on $\M_{++}$:
	we first embed the scalar divergence into a Hilbert space (Section~\ref{sec:scalar}),
	then establish a rearrangement lower bound using the Fack--Kosaki inequality (Section~\ref{sec:rearr}),
	and finally combine these ingredients with Minkowski's inequality (Section~\ref{sec:triangle}).
	Section~\ref{sec:GNS} transfers the von Neumann result back to tracial $C^*$-algebras via the GNS envelope
	and completes the proof of Theorem~\ref{thm:main-tracelog-Cstar}.
	Section~\ref{sec:qjsd} proves an integral representation for QJSD and deduces Theorem~\ref{thm:qjsd-Cstar}.
	Section~\ref{sec:hilbertian} establishes the non-Hilbertianity results in Theorem~\ref{thm:main-not-hilbertian}
	and provides an explicit numerical witness.
Finally, Section~\ref{sec:jensen-generators} discusses an operator-convex extension mechanism
for metricity of $\sqrt{J_{\tau,f}}$.

\medskip
\noindent\textbf{Acknowledgments. }The author is also grateful to his advisor, Lajos Moln\'ar, for suggesting this problem.	This work is supported by the China Scholarship Council, the Young Elite Scientists Sponsorship Program for PhD Students
(China Association for Science and Technology), and the Fundamental Research Funds for the Central Universities at Xian Jiaotong University
(Grant No.~xzy022024045). 
	
	% ============================================================
	\section{Von Neumann preliminaries: determinants and rearrangements}\label{sec:prelim}
	% ============================================================

	\subsection{Fuglede--Kadison determinant}
	
	For an invertible element $x\in \M^\times$, the Fuglede--Kadison determinant is defined by
	\[
	\Delta(x):=\exp\big(\tau(\log|x|)\big)\in (0,\infty).
	\]
	In the original work \cite{FK52} (for finite factors) one shows that $\Delta:\M^\times\to(0,\infty)$ is a continuous group homomorphism,
	hence multiplicative (see \cite[Theorem~1(1)]{FK52}):
	\begin{equation}\label{eq:FK-mult}
		\Delta(xy)=\Delta(x)\Delta(y)\qquad (x,y\in \M^\times).
	\end{equation}
	The same conclusions extend to general finite von Neumann algebras with faithful normal tracial states,
	e.g.\ by reduction to factors via central decomposition; see \cite{Dix57,Har13}.
	
	When $x\in\M^\times$, we have $\log\Delta(x)=\tau(\log|x|)$.
	Moreover, for $X\in\Minvpos$ and $r\in\RR$,
	\begin{equation}\label{eq:Delta-power}
		\Delta(X^r)=\Delta(X)^r,
	\end{equation}
	since $\log\Delta(X^r)=\tau(\log(X^r))=r\,\tau(\log X)=r\log\Delta(X)$.
	
	\subsection{Generalized $s$-numbers and a trace formula}\label{subsec:mus}
	
	We recall $\tau$-measurability and generalized $s$-numbers in the sense of Fack--Kosaki \cite{FK86}.
	Let $\M\subset B(H)$ be a von Neumann algebra acting on a complex Hilbert space $H$, and let
	$\tau:\M_{+}\to[0,\infty]$ be a faithful normal trace.
	In our applications $\tau$ will be \emph{finite} and normalized, i.e.\ $\tau(\Id)=1$.
	
	Let $T$ be a densely-defined closed operator on $H$ (possibly unbounded) affiliated with $\M$, and write $\Dom(T)$ for its domain.
	Following \cite[Def.~1.2]{FK86}, $T$ is called \emph{$\tau$-measurable} if for every $\varepsilon>0$ there exists a projection $E\in\M$ such that
	\[
	E(H)\subset\Dom(T)\qquad\text{and}\qquad \tau(\Id-E)\le\varepsilon.
	\]
	Since $\tau$ is finite in our setting, the algebra $\widetilde{\M}$ of $\tau$-measurable operators coincides with the set of all densely-defined closed operators
	affiliated with $\M$; see \cite[p.~271, after Def.~1.2]{FK86}.
	
	If $T$ is affiliated with $\M$, write $T=U|T|$ for its polar decomposition and let $E^{|T|}$ denote the spectral measure of $|T|$.
	For $\lambda\ge 0$, set
	\begin{equation*}
		d_{|T|}(\lambda):=\tau\!\left(E^{|T|}\big((\lambda,\infty)\big)\right).
	\end{equation*}
	The \emph{generalized $s$-numbers} of $T$ are defined for $t>0$ by
	\begin{equation}\label{eq:mu-def}
		\mu_t(T):=\inf\{\lambda\ge 0:\ d_{|T|}(\lambda)\le t\},
	\end{equation}
	see \cite[Def.~2.1]{FK86}. For a bounded positive operator $X\in\M_{+}$ we write $\mu_X(t):=\mu_t(X)$ for $t\in(0,1]$.
	
	Assume now that $\tau(\Id)=1$.
	\begin{remark}[Normalization of the trace]\label{rem:trace-normalization}
		Throughout Sections~\ref{sec:prelim}--\ref{sec:triangle} we work with a \emph{tracial state},
		so $\tau(\Id)=1$ and the generalized singular value parameter $t$ naturally ranges in $(0,1]$.
		If instead $\tau$ is a finite faithful normal trace with $\tau(\Id)=\theta\in(0,\infty)$,
		one may normalize it by $\tau_0:=\theta^{-1}\tau$.
		All rearrangement identities then hold with $\tau_0$ (and hence with integration range $(0,1]$),
		or equivalently one may keep $\tau$ and replace $\int_0^1(\cdot)\,dt$ by $\int_0^\theta(\cdot)\,dt$.
		Since metricity is unchanged by rescaling the trace (it only rescales the distance by a constant factor),
		we restrict to $\tau(\Id)=1$ for notational simplicity.
	\end{remark}
	\begin{remark}[Endpoint behavior of $\mu_X$]\label{rem:mu-endpoint}
		Assume $\tau(\Id)=1$ and let $X\in\M_{+}$ be bounded.
		Then $t\mapsto \mu_X(t)$ is decreasing and right-continuous on $(0,1]$.
		Moreover, one always has $\mu_X(1)=0$ (since $d_X(\lambda)\le 1$ for all $\lambda\ge 0$).
		
		When $X\in\M_{++}$, there exists $m>0$ such that $\mu_X(t)\ge m$ for all $t\in(0,1)$.
		Accordingly, integrals of the form $\int_0^1 \varphi(\mu_X(t))\,dt$ are understood as Lebesgue integrals on $(0,1)$,
		and the value at $t=1$ is irrelevant.
	\end{remark}
	
	Let $f:[0,\infty)\to\RR$ be continuous and increasing with $f(0)=0$.
	For every $\tau$-measurable operator $T$, Fack--Kosaki proved the trace formula
	\begin{equation}\label{eq:FK-Cor28}
		\tau\big(f(|T|)\big)=\int_0^\infty f\big(\mu_t(T)\big)\,\dd t,
	\end{equation}
	see \cite[Cor.~2.8]{FK86}.  Moreover, if $X\in\M_{+}$ is bounded, then $\mu_t(X)=0$ for all $t>\tau(\Id)=1$
	(see \cite[Lem.~2.6]{FK86}), so \eqref{eq:FK-Cor28} reduces to
	\begin{equation}\label{eq:trace-rearr}
		\tau(f(X))=\int_0^1 f\big(\mu_X(t)\big)\,\dd t.
	\end{equation}
	
\begin{remark}[Shifting constants]\label{rem:shift-constants}
	Assume $\tau(\Id)=1$ and let $X\in\M_{+}$ be bounded.
	If $f:[0,\infty)\to\RR$ is continuous and increasing (not necessarily satisfying $f(0)=0$), then
	\begin{equation}\label{eq:trace-rearr-shifted}
		\tau(f(X))=\int_0^1 f\big(\mu_X(t)\big)\,\dd t.
	\end{equation}
\end{remark}

\medskip
Finally, if $X\in\M_{++}$ then $X\ge m\Id$ for some $m>0$, hence $\mu_X(t)\ge m$ for all $t\in(0,1)$
(cf.\ Remark~\ref{rem:mu-endpoint}).
Define a continuous increasing function $g_m:[0,\infty)\to\RR$ by
\[
g_m(u):=
\begin{cases}
	0, & 0\le u\le m,\\[0.2em]
	\log(u/m), & u\ge m.
\end{cases}
\]
Then $g_m(0)=0$ and, since $\sigma(X)\subset[m,\|X\|]$, functional calculus gives
\[
g_m(X)=\log X-(\log m)\Id,
\qquad
g_m(\mu_X(t))=\log(\mu_X(t))-\log m\quad \text{for a.e. }t\in(0,1).
\]
Applying \eqref{eq:trace-rearr} to $g_m$ yields
\[
\tau(\log X)-\log m
=\tau(g_m(X))
=\int_0^1 g_m(\mu_X(t))\,\dd t
=\int_0^1 \bigl(\log(\mu_X(t))-\log m\bigr)\,\dd t,
\]
and therefore
\begin{equation}\label{eq:trace-log-mu}
	\tau(\log X)=\int_{(0,1)} \log\big(\mu_X(t)\big)\,\dd t,\qquad X\in\M_{++}.
\end{equation}
Here the integral is understood in the Lebesgue sense; the integrand is finite for $t\in(0,1)$ since
$\mu_X(t)\ge m>0$ on $(0,1)$ (cf.\ Remark~\ref{rem:mu-endpoint}).
	
	\medskip
	\noindent\textbf{Fack--Kosaki inequality.}
	A key inequality we will use is \cite[Theorem~4.4(iii)]{FK86}: for $X,Y\in\M_{+}$ and every continuous convex \emph{increasing} function $f:[0,\infty)\to\RR$,
	\begin{equation}\label{eq:FK}
		\int_0^u f\big(\mu_{X+Y}(t)\big)\,\dd t
		\le
		\int_0^u f\big(\mu_X(t)+\mu_Y(t)\big)\,\dd t,
		\qquad 0<u\le 1.
	\end{equation}
	
	% ============================================================
	\section{The trace-log distance on $\Minvpos$}\label{sec:tracelog-basic}
	% ============================================================
	
	\subsection{Arithmetic and geometric means}
	
	For $A,B\in\Minvpos$ define the arithmetic mean and geometric mean by
	\[
	A\nablaop B:=\frac{A+B}{2},
	\qquad
	A\sharpop B:=A^{1/2}\big(A^{-1/2}BA^{-1/2}\big)^{1/2}A^{1/2}.
	\]
	The geometric mean is symmetric and satisfies the arithmetic--geometric mean inequality
	\begin{equation}\label{eq:ag}
		A\sharpop B \le A\nablaop B.
	\end{equation}
	
	\subsection{Equivalent forms of the distance}
	
	\begin{lem}[Geometric term as an averaged trace-log]\label{lem:geo-trace}
		For $A,B\in\Minvpos$,
		\[
		\tau(\log(A\sharpop B))=\frac12\tau(\log A)+\frac12\tau(\log B).
		\]
		Consequently,
		\begin{equation}\label{eq:alt-form}
			d_\tau(A,B)^2
			=\tau\!\left(\log(A\nablaop B)\right)-\frac12\tau(\log A)-\frac12\tau(\log B)
			=\tau\!\big(\log(A\nablaop B)-\log(A\sharpop B)\big).
		\end{equation}
	\end{lem}
	
	\begin{proof}
		Using $\tau(\log X)=\log\Delta(X)$ and multiplicativity \eqref{eq:FK-mult},
		\[
		\Delta(A\sharpop B)=\Delta(A^{1/2})\,\Delta\!\big((A^{-1/2}BA^{-1/2})^{1/2}\big)\,\Delta(A^{1/2}).
		\]
		Since $\Delta(A^{1/2})^2=\Delta(A)$ and $\Delta(X^{1/2})=\Delta(X)^{1/2}$ on $\Minvpos$ (by \eqref{eq:Delta-power}), we get
		\[
		\Delta(A\sharpop B)=\Delta(A)\cdot \Delta(A^{-1/2}BA^{-1/2})^{1/2}.
		\]
		Again by multiplicativity,
		\[
		\Delta(A^{-1/2}BA^{-1/2})
		=\Delta(A^{-1/2})\,\Delta(B)\,\Delta(A^{-1/2})
		=\Delta(A)^{-1}\Delta(B).
		\]
		Hence $\Delta(A\sharpop B)=\Delta(A)^{1/2}\Delta(B)^{1/2}$, and taking $\log$ yields the claimed identity.
		Substituting into \eqref{eq:d_tau} gives \eqref{eq:alt-form}.
	\end{proof}
	
	\subsection{Basic metric axioms (except triangle inequality)}
	
	\begin{prop}[Nonnegativity, symmetry, definiteness]\label{prop:basic}
		For all $A,B\in\Minvpos$:
		\begin{enumerate}[label=\textup{(\arabic*)}, leftmargin=2em]
			\item $d_\tau(A,B)\ge 0$.
			\item $d_\tau(A,B)=d_\tau(B,A)$.
			\item $d_\tau(A,B)=0$ if and only if $A=B$.
		\end{enumerate}
	\end{prop}
	
	\begin{proof}
		Symmetry is immediate from \eqref{eq:d_tau}.
		Nonnegativity follows from \eqref{eq:ag} and operator monotonicity of $\log$:
		$\log(A\sharpop B)\le \log(A\nablaop B)$, hence $d_\tau(A,B)^2\ge 0$ by \eqref{eq:alt-form}.
		
		For definiteness, assume $d_\tau(A,B)=0$. Then $\log(A\nablaop B)-\log(A\sharpop B)\ge 0$ and its trace is $0$,
		so by faithfulness it must be $0$. Thus $A\nablaop B=A\sharpop B$.
		Let $X:=A^{-1/2}BA^{-1/2}\in\Minvpos$. Conjugating gives $(\Id+X)/2=X^{1/2}$,
		so $(X^{1/2}-\Id)^2=0$. Since $X^{1/2}-\Id$ is self-adjoint, $X^{1/2}=\Id$, hence $X=\Id$ and $A=B$.
		The converse is immediate.
	\end{proof}
	
	\subsection{Congruence invariance and scaling}
	
	\begin{lem}[Congruence invariance]\label{lem:congruence}
		For every invertible $S\in \M^\times$ and all $A,B\in\Minvpos$,
		\[
		d_\tau(S^*AS,S^*BS)=d_\tau(A,B).
		\]
	\end{lem}
	
	\begin{proof}
		By \eqref{eq:alt-form}, it suffices to show that for $X\in\Minvpos$,
		\begin{equation}\label{eq:cong-trace-log}
			\tau(\log(S^*XS))=\tau(\log X)+2\log\Delta(S).
		\end{equation}
		Indeed, applying \eqref{eq:cong-trace-log} to $X=A\nablaop B$, $X=A$, and $X=B$ makes the additive constants cancel.
		
		To prove \eqref{eq:cong-trace-log}, use $\tau(\log Y)=\log\Delta(Y)$ and multiplicativity:
		\[
		\tau(\log(S^*XS))=\log\Delta(S^*XS)=\log\Delta(S^*)+\log\Delta(X)+\log\Delta(S).
		\]
		Since $\Delta(S^*)=\Delta(S)$ (via polar decomposition and traciality), \eqref{eq:cong-trace-log} follows.
	\end{proof}
	
	\begin{lem}[Positive scalar homogeneity]\label{lem:scaling}
		For every $c>0$ and all $A,B\in\Minvpos$,
		\[
		d_\tau(cA,cB)=d_\tau(A,B).
		\]
	\end{lem}
	
	\begin{proof}
		Use $\log(cX)=(\log c)\Id+\log X$ in \eqref{eq:d_tau} and cancel the constant terms.
	\end{proof}
	
	\begin{remark}[Reduction to a base point]\label{rem:basepoint}
		By Lemma~\ref{lem:congruence}, for $A,B\in\Minvpos$,
		\[
		d_\tau(A,B)=d_\tau\big(\Id,\,A^{-1/2}BA^{-1/2}\big).
		\]
		Thus the triangle inequality reduces to bounds for $d_\tau(\Id,\cdot)$ and comparisons along rearrangements.
	\end{remark}
	
	% ============================================================
	\section{Scalar divergence as a Hilbert space distance}\label{sec:scalar}
	% ============================================================
	
	Define for $x,y>0$ the scalar divergence
	\begin{equation}\label{eq:scalar-delta}
		\delta_s(x,y)^2 := \log\!\Big(\frac{x+y}{2}\Big)-\frac12\log x-\frac12\log y,
		\qquad
		\delta_s(x,y):=\sqrt{\delta_s(x,y)^2}.
	\end{equation}
	
	\begin{lem}[Integral representation]\label{lem:integral-repr}
		For $x,y>0$,
		\begin{equation}\label{eq:integral-repr}
	\delta_s(x,y)^2
	=\frac12\int_0^\infty \big(e^{-rx/2}-e^{-ry/2}\big)^2\,\frac{\dd r}{r}.
		\end{equation}
	\end{lem}
	
	\begin{proof}
		Use the classical Laplace representation for $\log u$ (valid for all $u>0$),
		\[
		\log u=\int_0^\infty \frac{e^{-s}-e^{-us}}{s}\,\dd s,
		\]
		as an improper integral.
		Apply it to $u=(x+y)/2$, $u=x$, and $u=y$, combine the terms, and factor the numerator.
		
		\medskip
		\noindent\emph{Convergence check.}
		As $r\downarrow 0$ we have $e^{-rx/2}-e^{-ry/2}=\frac{r}{2}(y-x)+O(r^2)$, hence
		$\big(e^{-rx/2}-e^{-ry/2}\big)^2/r = O(r)$, which is integrable near $0$.
		As $r\to\infty$ the integrand decays exponentially.
		Therefore the improper integral in \eqref{eq:integral-repr} is finite.
		
	\end{proof}
	
	\begin{cor}[Scalar triangle inequality]\label{cor:scalar-metric}
		The function $\delta_s$ is a metric on $(0,\infty)$.
	\end{cor}
	
	\begin{proof}
		Fix $x_0>0$ and define $\Phi_{x_0}:(0,\infty)\to L^2((0,\infty),\dd r/r)
		$ by
		\[
	\Phi_{x_0}(x)(r):=\frac{1}{\sqrt2}\,\big(e^{-rx/2}-e^{-r x_0/2}\big).
		\]
		Lemma~\ref{lem:integral-repr} gives
		\[
	\delta_s(x,y)^2=\|\Phi_{x_0}(x)-\Phi_{x_0}(y)\|_{L^2((0,\infty),\dd r/r)}^2,
		\]
		so $\delta_s$ is the pullback of a Hilbert space metric and satisfies the triangle inequality.
	For definiteness, assume $\delta_s(x,y)=0$. Then by \eqref{eq:integral-repr},
	$e^{-rx/2}=e^{-ry/2}$ for a.e.\ $r>0$.
	If $x\neq y$, the function $r\mapsto e^{-rx/2}-e^{-ry/2}$ is real-analytic on $(0,\infty)$ and not identically zero,
	hence its zero set is discrete and cannot have positive measure. Therefore $x=y$.
	
	\end{proof}
	
	% ============================================================
	\section{A rearrangement lower bound for \texorpdfstring{$\tau(\log(X+Y))$}{tau(log(X+Y))}}\label{sec:rearr}
	% ============================================================
	
	\begin{prop}[Rearrangement lower bound]\label{prop:log-sum-lower}
		Let $X,Y\in\Minvpos$. Then $X+Y\in\Minvpos$ and
		\begin{equation}\label{eq:log-sum-lower}
			\tau(\log(X+Y))
			\ge
			\int_{(0,1)} \log\big(\mu_X(t)+\mu_Y(t)\big)\,\dd t.
		\end{equation}
	\end{prop}
	\begin{proof}
		Since $X,Y$ are invertible, there exist $m_X,m_Y>0$ such that $X\ge m_X\Id$ and $Y\ge m_Y\Id$.
		Then $X+Y\ge (m_X+m_Y)\Id$, hence $X+Y$ is invertible.
		
		Set $m:=m_X+m_Y$. Define
		\[
		f_m(s):=
		\begin{cases}
			-\log m+1, & 0\le s\le m,\\[0.3em]
			-\log s + \dfrac{s}{m}, & s\ge m.
		\end{cases}
		\]
	\medskip
	\noindent\emph{Remark (constant shift).}
	If one prefers a formulation of \eqref{eq:FK} assuming $f(0)=0$, apply it to
	$\widetilde f_m:=f_m-f_m(0)$. Since we take $u=1$, the constant term cancels on both sides.
		Then $f_m$ is continuous, convex, and increasing on $[0,\infty)$.
		Apply \eqref{eq:FK} with $u=1$ and $f=f_m$:
		\[
		\int_0^1 f_m(\mu_{X+Y}(t))\,\dd t
		\le
		\int_0^1 f_m(\mu_X(t)+\mu_Y(t))\,\dd t.
		\]
Since $X\ge m_X\Id$, we have $E^X((\lambda,\infty))=\Id$ for every $0\le \lambda<m_X$, hence
$d_X(\lambda)=\tau(\Id)=1$ for $\lambda<m_X$. By \eqref{eq:mu-def} this implies $\mu_X(t)\ge m_X$ for all $t\in(0,1)$.
Similarly $\mu_Y(t)\ge m_Y$ and $\mu_{X+Y}(t)\ge m_X+m_Y=m$ for all $t\in(0,1)$.
	For a.e.\ $t\in(0,1)$ we have $\mu_{X+Y}(t)\ge m$ and $\mu_X(t)+\mu_Y(t)\ge m$, hence
	$f_m(z)=-\log z+\frac{z}{m}$ on the relevant range.
	Moreover, by \eqref{eq:trace-rearr} applied to $f(s)=s$,
	\[
	\int_0^1 \mu_{X+Y}(t)\,\dd t=\tau(X+Y)=\tau(X)+\tau(Y)
	=\int_0^1\big(\mu_X(t)+\mu_Y(t)\big)\,\dd t.
	\]
	Hence the linear terms $\frac{1}{m}(\cdot)$ cancel, and we obtain
	\[
	\int_0^1 \log\big(\mu_{X+Y}(t)\big)\,\dd t
	\ge
	\int_0^1 \log\big(\mu_X(t)+\mu_Y(t)\big)\,\dd t.
	\]
		Finally, apply \eqref{eq:trace-log-mu} to $X+Y$ to obtain \eqref{eq:log-sum-lower}.
	\end{proof}
	
	% ============================================================
	\section{Triangle inequality for the trace-log distance on \texorpdfstring{$\M_{++}$}{M++}}\label{sec:triangle}
	% ============================================================
	
	\subsection{$d_\tau(\Id,\cdot)$ as an $L^2$-norm along rearrangements}
	
	\begin{lem}[$d_\tau(\Id,\cdot)$ via generalized singular values]\label{lem:d1-mu}
		For every $X\in\Minvpos$,
		\[
	d_\tau(\Id,X)^2=\int_{(0,1)} \delta_s\big(1,\mu_X(t)\big)^2\,\dd t,
		\]
	where the integrand is taken for a.e. $t\in(0,1)$,	hence
		\[
		d_\tau(\Id,X)=\big\|\delta_s(1,\mu_X)\big\|_{L^2(0,1)}.
		\]
	\end{lem}
	
\begin{proof}
	By \eqref{eq:d_tau} with $A=\Id$,
	\[
	d_\tau(\Id,X)^2=\tau\!\left(\log\!\Big(\frac{\Id+X}{2}\Big)\right)-\frac12\,\tau(\log X).
	\]
	Since $u\mapsto \log\!\big(\frac{1+u}{2}\big)$ is continuous and increasing on $[0,\infty)$, we may apply
	\eqref{eq:trace-rearr-shifted} to obtain
	\[
	\tau\!\left(\log\!\Big(\frac{\Id+X}{2}\Big)\right)
	=\int_0^1 \log\!\Big(\frac{1+\mu_X(t)}{2}\Big)\,\dd t.
	\]
	Moreover, since $X\in\M_{++}$, \eqref{eq:trace-log-mu} gives
	\[
	\tau(\log X)=\int_0^1 \log\big(\mu_X(t)\big)\,\dd t.
	\]
	Substituting these into the expression for $d_\tau(\Id,X)^2$ yields
	\[
	d_\tau(\Id,X)^2
	=\int_0^1\left(\log\!\Big(\frac{1+\mu_X(t)}{2}\Big)-\frac12\log\big(\mu_X(t)\big)\right)\dd t
	=\int_0^1 \delta_s\bigl(1,\mu_X(t)\bigr)^2\,\dd t,
	\]
	by the definition \eqref{eq:scalar-delta} of the scalar divergence. Since $X\in\M_{++}$, there exists $m>0$ such that $X\ge m\Id$, hence $\mu_X(t)\ge m$ for all $t\in(0,1)$.
	Therefore $\delta_s(1,\mu_X(t))$ is finite for a.e.\ $t\in(0,1)$ (cf.\ Remark~\ref{rem:mu-endpoint}).
\end{proof}

	\subsection{A lower bound for $d_\tau(S,T)$}
	
	\begin{lem}[Fack--Kosaki lower bound for $d_\tau(S,T)$]\label{lem:dST-lower}
		For all $S,T\in\Minvpos$,
		\begin{equation}\label{eq:dST-lower}
			d_\tau(S,T)^2
			\ge
			\int_0^1 \delta_s\big(\mu_S(t),\mu_T(t)\big)^2\,\dd t,
		\end{equation}
where the integrand is taken for a.e. $t\in(0,1)$. Hence
\[
\big\|\delta_s(\mu_S,\mu_T)\big\|_{L^2(0,1)}\le d_\tau(S,T).
\]
	\end{lem}
	
	\begin{proof}
		Using \eqref{eq:d_tau} and $\log(\frac{S+T}{2})=\log(S+T)-(\log 2)\Id$,
		Proposition~\ref{prop:log-sum-lower} yields
		\[
		\tau\!\left(\log\!\Big(\frac{S+T}{2}\Big)\right)
		\ge
		\int_0^1 \log\!\Big(\frac{\mu_S(t)+\mu_T(t)}{2}\Big)\,\dd t.
		\]
		Moreover, \eqref{eq:trace-log-mu} gives $\tau(\log S)=\int_0^1\log\mu_S$ and similarly for $T$.
		Subtracting and comparing with \eqref{eq:scalar-delta} gives \eqref{eq:dST-lower}.
	\end{proof}
	
	\subsection{Triangle inequality}
	
	\begin{prop}[Triangle inequality on $\Minvpos$]\label{prop:triangle}
		For all $A,B,C\in\Minvpos$,
		\[
		d_\tau(A,C)\le d_\tau(A,B)+d_\tau(B,C).
		\]
	\end{prop}
	
	\begin{proof}
		By congruence invariance (Lemma~\ref{lem:congruence}) with $S=A^{-1/2}$, it suffices to prove
		\begin{equation}\label{eq:triangle-reduced}
			d_\tau(\Id,T)\le d_\tau(\Id,S)+d_\tau(S,T)\qquad(S,T\in\Minvpos).
		\end{equation}
		Fix $S,T\in\Minvpos$.
		By Lemma~\ref{lem:d1-mu},
		\[
		d_\tau(\Id,T)=\|\delta_s(1,\mu_T)\|_{L^2(0,1)},
		\qquad
		d_\tau(\Id,S)=\|\delta_s(1,\mu_S)\|_{L^2(0,1)}.
		\]
		By Lemma~\ref{lem:dST-lower}, $\|\delta_s(\mu_S,\mu_T)\|_{L^2(0,1)}\le d_\tau(S,T)$.

		Since $t\mapsto \mu_S(t)$ and $t\mapsto \mu_T(t)$ are decreasing and right-continuous on $(0,1)$, they are measurable.
		As $\delta_s$ is continuous on $(0,\infty)^2$, the functions
		$t\mapsto \delta_s(1,\mu_S(t))$ and $t\mapsto \delta_s(\mu_S(t),\mu_T(t))$ are measurable;
		moreover they belong to $L^2(0,1)$ by Lemma~\ref{lem:d1-mu} and Lemma~\ref{lem:dST-lower}.
	Since $\delta_s$ is a scalar metric (Corollary~\ref{cor:scalar-metric}), for a.e.\ $t\in(0,1)$,
	\[
	\delta_s(1,\mu_T(t))\le \delta_s(1,\mu_S(t))+\delta_s(\mu_S(t),\mu_T(t)).
	\]
		Taking $L^2$-norms and applying Minkowski's inequality yields \eqref{eq:triangle-reduced}.
	\end{proof}
	
	\begin{thm}[Trace-log metric on a finite von Neumann algebra]\label{thm:main-vN}
		Let $(\M,\tau)$ be a finite von Neumann algebra with a faithful normal tracial state.
		Then $d_\tau$ is a metric on $\M_{++}$.
	\end{thm}
	
	\begin{proof}
		Proposition~\ref{prop:basic} gives symmetry, nonnegativity, and definiteness.
		Proposition~\ref{prop:triangle} gives the triangle inequality.
	\end{proof}
	
	% ============================================================
	\section{Transfer to tracial \texorpdfstring{$C^*$}{C*}-algebras via the GNS von Neumann envelope}\label{sec:GNS}
	% ============================================================
	
	\begin{lem}[The induced trace on the von Neumann envelope]\label{lem:tilde-tau}
		Let $(\A,\tau)$ be a unital $C^*$-algebra with a faithful tracial state and let $(H_\tau,\pi_\tau,\xi_\tau)$ be the GNS triple.
		Set $\M:=\pi_\tau(\A)''\subset B(H_\tau)$ and define $\tilde\tau(x):=\langle x\xi_\tau,\xi_\tau\rangle$ for $x\in\M$.
		Then $\tilde\tau$ is a faithful normal tracial state on $\M$.
	\end{lem}
	
	\begin{proof}
		Normality is immediate since $\tilde\tau$ is the restriction to $\M$ of a vector state on $B(H_\tau)$.
		
		For traciality, for $a,b\in\A$ we have
		\[
		\tilde\tau(\pi_\tau(a)\pi_\tau(b))=\langle \pi_\tau(ab)\xi_\tau,\xi_\tau\rangle=\tau(ab)=\tau(ba)=\tilde\tau(\pi_\tau(b)\pi_\tau(a)).
		\]
		Since $\pi_\tau(\A)$ is strongly dense in $\M$ and $\tilde\tau$ is normal, this extends to all of $\M$.
		
For faithfulness, let $x\in\M_{+}$ and assume $\tilde\tau(x)=0$.
Then
\[
\|x^{1/2}\xi_\tau\|^2=\langle x\xi_\tau,\xi_\tau\rangle=\tilde\tau(x)=0,
\]
hence $x^{1/2}\xi_\tau=0$.

We claim that $\xi_\tau$ is separating for $\M$.
To see this, recall that $H_\tau$ is the completion of $\A$ with inner product
$\langle a,b\rangle=\tau(b^*a)$ and $\xi_\tau$ is the class of $\Id$.
Define the \emph{right} action on the dense subspace $\pi_\tau(\A)\xi_\tau$ by
\[
\rho_\tau(a)\,\pi_\tau(b)\xi_\tau:=\pi_\tau(ba)\xi_\tau\qquad(a,b\in\A).
\]
For $a,b\in\A$ we have
\[
\|\rho_\tau(a)\pi_\tau(b)\xi_\tau\|^2
=\|\pi_\tau(ba)\xi_\tau\|^2
=\tau\big((ba)^*(ba)\big)
=\tau\big(a^*b^*ba\big)
=\tau\big(b^*b\,aa^*\big)
\le \|a\|^2\,\tau(b^*b).
\]\medskip
\noindent\emph{Well-definedness.}
If $\pi_\tau(b)\xi_\tau=0$, then $\|\pi_\tau(b)\xi_\tau\|^2=\tau(b^*b)=0$, hence $b=0$ by faithfulness of $\tau$.
Thus the definition of $\rho_\tau(a)$ on $\pi_\tau(\A)\xi_\tau$ is unambiguous.

Thus $\rho_\tau(a)$ extends to a bounded operator on $H_\tau$ with $\|\rho_\tau(a)\|\le\|a\|$.
\noindent
A direct computation on the dense subspace $\pi_\tau(\A)\xi_\tau$ shows that $\rho_\tau(a)$ commutes with each $\pi_\tau(c)$:
for $a,b,c\in\A$,
$\pi_\tau(c)\rho_\tau(a)\pi_\tau(b)\xi_\tau=\pi_\tau(cba)\xi_\tau=\rho_\tau(a)\pi_\tau(c)\pi_\tau(b)\xi_\tau$.
Hence $\rho_\tau(\A)\subset \pi_\tau(\A)' \subset \M'$.
Moreover,
\[
\rho_\tau(\A)\xi_\tau=\{\pi_\tau(a)\xi_\tau:\ a\in\A\}
\]
is dense in $H_\tau$, hence $\xi_\tau$ is cyclic for $\M'$.
Therefore $\xi_\tau$ is separating for $\M$: if $y\in\M$ and $y\xi_\tau=0$, then
$y z\xi_\tau = z y\xi_\tau=0$ for all $z\in\M'$, so $y$ vanishes on the dense subspace $\M'\xi_\tau$ and thus $y=0$.

Applying this to $y=x^{1/2}$ yields $x^{1/2}=0$, hence $x=0$.

	\end{proof}
	
	\begin{proof}[Proof of Theorem~\ref{thm:main-tracelog-Cstar}]
		Let $(\A,\tau)$ be a unital $C^*$-algebra with faithful trace and let $\M=\pi_\tau(\A)''$ be the von Neumann envelope.
		By Lemma~\ref{lem:tilde-tau}, $\tilde\tau$ is a faithful normal tracial state on $\M$.
		
		If $A\in\A_{++}$ then $\pi_\tau(A)\in\M_{++}$ and continuous functional calculus is respected:
		$\pi_\tau(\log A)=\log(\pi_\tau(A))$, and similarly for $\log(\frac{A+B}{2})$.
		Therefore for $A,B\in\A_{++}$,
		\[
		d_\tau(A,B)^2
		=
		\tilde\tau\!\left(\log\!\left(\frac{\pi_\tau(A)+\pi_\tau(B)}{2}\right)\right)
		-\frac12\tilde\tau(\log\pi_\tau(A))-\frac12\tilde\tau(\log\pi_\tau(B))
		=
		d_{\tilde\tau}(\pi_\tau(A),\pi_\tau(B))^2.
		\]
		Thus $d_\tau(A,B)=d_{\tilde\tau}(\pi_\tau(A),\pi_\tau(B))$.
		Since $d_{\tilde\tau}$ is a metric on $\M_{++}$ by Theorem~\ref{thm:main-vN},
		its restriction to $\pi_\tau(\A_{++})$ is a metric; faithfulness of $\pi_\tau$ transfers definiteness back to $\A_{++}$.
	\end{proof}
	
	% ============================================================
	\section{Quantum Jensen--Shannon divergence and an integral representation}\label{sec:qjsd}
	% ============================================================
	
	Throughout this section we work first in a finite von Neumann algebra $(\M,\tau)$ and then transfer to $\A$ via Section~\ref{sec:GNS}.
	Let $\etaop(x)=x\log x$ with $\etaop(0)=0$.
	
	\subsection{A differentiation identity}
	
	\begin{lem}[Differentiation formula]\label{lem:qjsd-deriv}
		Let $A\in\M_{+}$. The map $t\mapsto \tau(\etaop(A+t\Id))$ is continuously differentiable on $(0,\infty)$ and
		\[
		\frac{\dd}{\dd t}\tau(\etaop(A+t\Id))=\tau(\log(A+t\Id)+\Id)\qquad(t>0).
		\]
	\end{lem}
	
	\begin{proof}
		Let $E^A$ be the spectral measure of $A$ and define the finite measure $\nu_A(\Omega)=\tau(E^A(\Omega))$ on $[0,\|A\|]$.
		Then $\tau(f(A))=\int f(\lambda)\,\dd\nu_A(\lambda)$ for bounded Borel $f$.
		For $t>0$,
		\[
		\tau(\etaop(A+t\Id))=\int_{[0,\|A\|]} (\lambda+t)\log(\lambda+t)\,\dd\nu_A(\lambda).
		\]
	Fix $0<t_0<t_1<\infty$. For $t\in[t_0,t_1]$ and $\lambda\in[0,\|A\|]$,
	the derivative $\partial_t\big((\lambda+t)\log(\lambda+t)\big)=\log(\lambda+t)+1$
	is bounded in absolute value by a constant depending only on $t_0,t_1,\|A\|$.
	Therefore differentiation under the integral sign is justified by dominated convergence on $[t_0,t_1]$. The same domination argument applied to $\partial_t(\log(\lambda+t)+1)=1/(\lambda+t)$
	shows that the derivative depends continuously on $t$ on any compact subinterval of $(0,\infty)$.
	Hence $t\mapsto \tau(\etaop(A+t\Id))$ is $C^1$ on $(0,\infty)$.
	
	\end{proof}
	
	\subsection{Vanishing at infinity}
	
	\begin{lem}[Vanishing at infinity]\label{lem:qjsd-vanish}
		For all $A,B\in\M_{+}$,
		\[
		\lim_{t\to\infty} J_{\tau,\etaop}(A+t\Id,B+t\Id)=0.
		\]
	\end{lem}
	
	\begin{proof}
		Using $\etaop(cX)=c\,\etaop(X)+c(\log c)\,X$ and cancellation of affine terms in $J_{\tau,\etaop}$,
		one checks $J_{\tau,\etaop}(cA,cB)=c\,J_{\tau,\etaop}(A,B)$ for $c>0$.
		Thus
		\[
		J_{\tau,\etaop}(A+t\Id,B+t\Id)=t\,J_{\tau,\etaop}\!\left(\Id+\frac{A}{t},\,\Id+\frac{B}{t}\right).
		\]
Let $\phi(r):=\etaop(1+r)-r=(1+r)\log(1+r)-r$ for $r\ge 0$.
Fix $r_0\in(0,1)$. Since $\phi$ is $C^2$ on $[0,r_0]$ and $\phi(0)=\phi'(0)=0$,
there exists $C>0$ such that
\[
0\le \phi(r)\le C r^2\qquad (0\le r\le r_0).
\]
Choose $t_0>0$ so that $\|A\|/t_0\le r_0$ and $\|B\|/t_0\le r_0$.
Then for $t\ge t_0$, functional calculus gives
\[
\|\phi(A/t)\|\le C\|A\|^2/t^2,\qquad
\|\phi(B/t)\|\le C\|B\|^2/t^2,\qquad
\Big\|\phi\Big(\frac{A+B}{2t}\Big)\Big\|\le C\|A+B\|^2/(4t^2).
\]
Using $\etaop(\Id+X)=X+\phi(X)$ and the fact that affine terms cancel in $J_{\tau,\etaop}$,
we get
\[
J_{\tau,\etaop}\Big(\Id+\frac{A}{t},\,\Id+\frac{B}{t}\Big)
=\frac12\tau\!\Big(\phi\Big(\frac{A}{t}\Big)\Big)
+\frac12\tau\!\Big(\phi\Big(\frac{B}{t}\Big)\Big)
-\tau\!\Big(\phi\Big(\frac{A+B}{2t}\Big)\Big).
\]
Since $\tau$ is a state, $|\tau(X)|\le \|X\|$, hence
\[
\Big|J_{\tau,\etaop}\Big(\Id+\frac{A}{t},\,\Id+\frac{B}{t}\Big)\Big|
\le \frac12\Big\|\phi\Big(\frac{A}{t}\Big)\Big\|
+\frac12\Big\|\phi\Big(\frac{B}{t}\Big)\Big\|
+\Big\|\phi\Big(\frac{A+B}{2t}\Big)\Big\|
\le \frac{C'}{t^2}.
\]
Therefore $J_{\tau,\etaop}(A+t\Id,B+t\Id)=t\,J_{\tau,\etaop}(\Id+A/t,\Id+B/t)=O(t^{-1})\to 0$.
	\end{proof}
	
	\subsection{Integral representation and metricity}
	
	For $t>0$ and $A,B\in\M_{+}$ set
	\[
	d_{\tau,t}(A,B):=d_\tau(A+t\Id,B+t\Id).
	\]
	
	\begin{prop}[Integral representation]\label{prop:qjsd-integral}
		For all $A,B\in\M_{+}$,
		\begin{equation}\label{eq:qjsd-integral}
			J_{\tau,\etaop}(A,B)=\int_0^\infty d_{\tau,t}(A,B)^2\,\dd t.
		\end{equation}
	\end{prop}
	
	\begin{proof}
		Define $F(t):=J_{\tau,\etaop}(A+t\Id,B+t\Id)$ for $t\ge 0$.
		For $t>0$, Lemma~\ref{lem:qjsd-deriv} gives
		\[
		F'(t)
		=\frac12\tau(\log(A+t\Id)+\Id)+\frac12\tau(\log(B+t\Id)+\Id)
		-\tau\!\left(\log\!\left(\frac{A+B}{2}+t\Id\right)+\Id\right).
		\]
		The $\Id$-terms cancel, and the remaining expression equals $-d_{\tau,t}(A,B)^2$ by \eqref{eq:d_tau}.
		Hence $F'(t)=-d_{\tau,t}(A,B)^2$ for $t>0$.
		
		Fix $\varepsilon>0$. For $R>\varepsilon$,
		\[
		F(R)-F(\varepsilon)=\int_\varepsilon^R F'(t)\,\dd t=-\int_\varepsilon^R d_{\tau,t}(A,B)^2\,\dd t.
		\]
		Letting $R\to\infty$ and using Lemma~\ref{lem:qjsd-vanish} gives
		\[
		F(\varepsilon)=\int_\varepsilon^\infty d_{\tau,t}(A,B)^2\,\dd t.
		\]
	The map $t\mapsto d_{\tau,t}(A,B)^2$ is Borel measurable on $(0,\infty)$ (indeed, it is norm-continuous on $[\varepsilon,\infty)$ for every $\varepsilon>0$ by continuous functional calculus for $\log$).
	
		Moreover, $F(t)=J_{\tau,\etaop}(A+t\Id,B+t\Id)$ is continuous at $t=0$
		since $\etaop$ extends continuously to $[0,\infty)$ and functional calculus is norm-continuous for continuous functions.
		
	Since $d_{\tau,t}(A,B)^2\ge 0$, the functions $t\mapsto \ind_{(\varepsilon,\infty)}(t)\,d_{\tau,t}(A,B)^2$
	increase pointwise to $d_{\tau,t}(A,B)^2$ as $\varepsilon\downarrow 0$.
	Hence, by the monotone convergence theorem,
	\[
	\lim_{\varepsilon\downarrow 0}\int_\varepsilon^\infty d_{\tau,t}(A,B)^2\,\dd t
	=\int_0^\infty d_{\tau,t}(A,B)^2\,\dd t.
	\]
Finally, since $\etaop$ extends continuously to $[0,\infty)$ with $\etaop(0)=0$,
functional calculus is norm-continuous for the map $X\mapsto \etaop(X)$ on bounded sets.
Hence $t\mapsto \tau(\etaop(A+t\Id))$, $t\mapsto \tau(\etaop(B+t\Id))$, and
$t\mapsto \tau\!\big(\etaop(\tfrac{A+B}{2}+t\Id)\big)$ are continuous at $t=0$.
Therefore $F(\varepsilon)\to F(0)=J_{\tau,\etaop}(A,B)$ as $\varepsilon\downarrow 0$,
and combining this with the monotone convergence step above yields \eqref{eq:qjsd-integral}.
	\end{proof}
	
	\begin{prop}[QJSD is a metric in the von Neumann setting]\label{prop:qjsd-metric-vN}
		Let $A,B\in\M_{+}$ and set $D_{\tau,\etaop}:=\sqrt{J_{\tau,\etaop}}$.
		Then $D_{\tau,\etaop}$ is a metric on $\M_{+}$.
	\end{prop}
	
	\begin{proof}
		Symmetry is clear. Nonnegativity follows from \eqref{eq:qjsd-integral}.
		If $D_{\tau,\etaop}(A,B)=0$, then $d_{\tau,t}(A,B)=0$ for a.e.\ $t>0$, hence $A+t\Id=B+t\Id$ for such $t$ and therefore $A=B$.
		
		For the triangle inequality, let $A,B,C\in\M_{+}$.
		For each $t>0$, $d_{\tau,t}$ is the trace-log metric on $\M_{++}$ applied to $(A+t\Id,B+t\Id,C+t\Id)$,
		hence by Theorem~\ref{thm:main-vN},
		\[
		d_{\tau,t}(A,C)\le d_{\tau,t}(A,B)+d_{\tau,t}(B,C).
		\]
		Taking $L^2(0,\infty)$ norms and applying Minkowski yields
		\[
		D_{\tau,\etaop}(A,C)=\|d_{\tau,\bullet}(A,C)\|_{L^2(0,\infty)}
		\le \|d_{\tau,\bullet}(A,B)\|_{L^2}+\|d_{\tau,\bullet}(B,C)\|_{L^2}
		= D_{\tau,\etaop}(A,B)+D_{\tau,\etaop}(B,C).
		\]
	\end{proof}
	
	\begin{proof}[Proof of Theorem~\ref{thm:qjsd-Cstar}]
		As in Section~\ref{sec:GNS}, pass to the von Neumann envelope $\M=\pi_\tau(\A)''$ with induced faithful normal trace $\tilde\tau$.
		Functional calculus gives $\pi_\tau(\etaop(A))=\etaop(\pi_\tau(A))$ and $\tilde\tau(\pi_\tau(\etaop(A)))=\tau(\etaop(A))$.
		Hence $J_{\tau,\etaop}(A,B)=J_{\tilde\tau,\etaop}(\pi_\tau(A),\pi_\tau(B))$.
		By Proposition~\ref{prop:qjsd-metric-vN}, $D_{\tilde\tau,\etaop}$ is a metric; restricting to $\pi_\tau(\A_{+})$ and using faithfulness of $\pi_\tau$
		yields that $D_{\tau,\etaop}$ is a metric on $\A_{+}$.
	\end{proof}
	
	% ============================================================
	\section{Is the QJSD metric Hilbertian? Negative results in matrix dimensions}\label{sec:hilbertian}
	% ============================================================
	
	\subsection{Hilbertian metrics and conditional negative definiteness}
	
	\begin{defin}[Conditionally negative definite kernel / negative type]\label{def:cnd}
		Let $X$ be a set and $K:X\times X\to\RR$ be symmetric with $K(x,x)=0$.
		We say $K$ is \emph{conditionally negative definite} (CND) if for every $m\in\mathbb N$,
		every $x_1,\dots,x_m\in X$, and every $c_1,\dots,c_m\in\RR$ with $\sum_{i=1}^m c_i=0$, one has
		\[
		\sum_{i,j=1}^m c_i c_j\,K(x_i,x_j)\le 0.
		\]
		A metric $d$ on $X$ is \emph{Hilbertian} if there exist a Hilbert space $\mathcal H$ and an embedding $\Phi:X\to\mathcal H$
		such that $d(x,y)=\|\Phi(x)-\Phi(y)\|_{\mathcal H}$ for all $x,y\in X$.
	\end{defin}
	
	\begin{lem}[CND kernels produce Hilbert embeddings]\label{lem:cnd-embed}
		Let $K$ be a CND kernel on $X$.
		Then there exist a Hilbert space $\mathcal H$ and a map $\Phi:X\to\mathcal H$ such that
		\[
		K(x,y)=\|\Phi(x)-\Phi(y)\|_{\mathcal H}^{2}\qquad(x,y\in X).
		\]
		Consequently, a metric space $(X,d)$ is Hilbertian if and only if $d^2$ is CND.
	\end{lem}
	
	\begin{proof}
		Fix $x_0\in X$.
		Let $\mathcal V$ be the real vector space of finitely supported functions $u:X\to\RR$ with $\sum_{x\in X}u(x)=0$.
		Define a symmetric bilinear form on $\mathcal V$ by
		\[
		\langle u,v\rangle_0:=-\frac12\sum_{x,y\in X}u(x)v(y)\,K(x,y).
		\]
		CND of $K$ implies $\langle u,u\rangle_0\ge 0$ for all $u\in\mathcal V$.
		Quotient by the null space and complete to obtain a Hilbert space $\mathcal H$.
		For $x\in X$ let $u_x:=\delta_x-\delta_{x_0}\in\mathcal V$ and set $\Phi(x):=[u_x]\in\mathcal H$.
		A direct expansion shows $\|\Phi(x)-\Phi(y)\|^2=K(x,y)$.
		The final assertion follows by applying this to $K=d^2$ and observing that squared Hilbert distances are CND.
	\end{proof}
	
	\begin{lem}[Exponentials and CND]\label{lem:cnd-exp}
		Let $K:X\times X\to\RR$ be symmetric with $K(x,x)=0$.
		If $K$ is CND, then for every $\beta>0$ the kernel $\exp(-\beta K)$ is positive definite.
		Conversely, if $\exp(-\beta K)$ is positive definite for all $\beta>0$, then $K$ is CND.
	\end{lem}
	
	\begin{proof} This is a standard equivalence due to Schoenberg; see \cite{Sch38} (or \cite[Ch.~3]{BCR84}).
		Assume first that $K$ is CND.
		By Lemma~\ref{lem:cnd-embed}, $K(x,y)=\|\Phi(x)-\Phi(y)\|^2$ for some Hilbert embedding.
		For $u,v\in\mathcal H$,
		\[
		e^{-\beta\|u-v\|^2}=e^{-\beta\|u\|^2}e^{-\beta\|v\|^2}e^{2\beta\langle u,v\rangle}
		\]
		and $e^{2\beta\langle u,v\rangle}$ is positive definite as a power series with positive coefficients (tensor powers).
		Multiplying by the positive diagonal factors preserves positive definiteness.
		Hence $\exp(-\beta K)$ is positive definite.
		
		Conversely, assume $\exp(-\beta K)$ is positive definite for all $\beta>0$.
		Fix $x_1,\dots,x_m\in X$ and $c_1,\dots,c_m\in\RR$ with $\sum_i c_i=0$.
		Define
		\[
		F(\beta):=\sum_{i,j=1}^m c_i c_j\,e^{-\beta K(x_i,x_j)}\ge 0\qquad(\beta>0),
		\]
		and note $F(0)=\sum_{i,j}c_ic_j=(\sum_i c_i)^2=0$.
		Thus $F'(0^+)\ge 0$.
		But
		\[
		F'(\beta)=-\sum_{i,j} c_i c_j\,K(x_i,x_j)\,e^{-\beta K(x_i,x_j)},
		\]
		so $F'(0^+)=-\sum_{i,j}c_i c_j K(x_i,x_j)\ge 0$, i.e.\ $\sum_{i,j}c_ic_jK(x_i,x_j)\le 0$.
		Hence $K$ is CND.
	\end{proof}
	
	\subsection{Non-Hilbertianity on $M_n^{++}(\CC)$ for $n\ge 2$}
	
	In this subsection let $\A=M_n(\CC)$ and let $\tau=\frac1n\Tr$.

	\begin{prop}[QJSD is not Hilbertian on $M_n^{++}(\CC)$ for $n\ge 2$]\label{prop:not-hilbertian-cone}
		Let $n\ge 2$. The QJSD metric $D_{\tau,\etaop}$ on $M_n^{++}(\CC)$ is \emph{not} Hilbertian.
		Consequently, $D_{\tau,\etaop}$ is not Hilbertian on the full cone $M_n^{+}(\CC)$.
	\end{prop}
	
\begin{proof}
	Recall that a metric space $(X,d)$ is Hilbertian if and only if $d^2$ is CND
	(Lemma~\ref{lem:cnd-embed}). Thus it suffices to show that the kernel
	$J_{\tau,\etaop}$ is not CND.
	
	\medskip
	\noindent\emph{Step 1: the case $n=2$.}
	By Example~\ref{ex:explicit-integer}, there exist $X_1,\dots,X_m\in M_2^{++}(\mathbb{C})$ and $c_1,\dots,c_m\in\RR$
	with $\sum_{i=1}^m c_i=0$ such that
	\[
	\sum_{i,j=1}^m c_ic_j\,J_{\tau_2,\etaop}(X_i,X_j)>0,
	\qquad \text{where }\ \tau_2=\tfrac12\Tr \text{ on }M_2(\CC).
	\]
	Hence $J_{\tau_2,\etaop}$ is not CND on $M_2^{++}(\mathbb{C})$, and therefore the metric
	$D_{\tau_2,\etaop}=\sqrt{J_{\tau_2,\etaop}}$ is not Hilbertian on $M_2^{++}(\mathbb{C})$.
	
	\medskip
	\noindent\emph{Step 2: embedding into $M_n^{++}(\CC)$ for $n\ge 2$.}
	Let $n\ge 2$ and write $\tau_n=\tfrac1n\Tr$ on $M_n(\CC)$.
	Define the block embedding $\iota:M_2^{++}(\mathbb{C})\to M_n^{++}(\CC)$ by
	\[
	\iota(X):=X\oplus \Id_{n-2}.
	\]
	Since $\etaop(1)=1\cdot\log 1=0$ and $\etaop$ acts blockwise on block-diagonal matrices, we have
	\[
	\etaop(\iota(X))=\etaop(X)\oplus 0_{n-2},
	\qquad
	\etaop\!\left(\frac{\iota(X)+\iota(Y)}{2}\right)
	=
	\etaop\!\left(\frac{X+Y}{2}\right)\oplus 0_{n-2}.
	\]
	Therefore, for all $X,Y\in M_2^{++}(\mathbb{C})$,
	\[
	J_{\tau_n,\etaop}(\iota(X),\iota(Y))
	=
	\frac1n\Tr\!\left(
	\frac12\etaop(X)+\frac12\etaop(Y)-\etaop\!\left(\frac{X+Y}{2}\right)
	\right)
	=
	\frac{2}{n}\,J_{\tau_2,\etaop}(X,Y).
	\]
	Now set $Y_i:=\iota(X_i)\in M_n^{++}(\CC)$ and use the same coefficients $c_i$ as in Step~1:
	\[
	\sum_{i,j=1}^m c_ic_j\,J_{\tau_n,\etaop}(Y_i,Y_j)
	=
	\frac{2}{n}\sum_{i,j=1}^m c_ic_j\,J_{\tau_2,\etaop}(X_i,X_j)
	>0.
	\]
	Hence $J_{\tau_n,\etaop}$ is not CND on $M_n^{++}(\CC)$, so $D_{\tau_n,\etaop}$ is not Hilbertian on $M_n^{++}(\CC)$.
	
	\medskip
Finally, if $D_{\tau_n,\etaop}$ were Hilbertian on $M_n(\CC)_+$, there would exist an embedding $\Phi:M_n(\CC)_+\to\mathcal H$
into a Hilbert space with $D_{\tau_n,\etaop}(X,Y)=\|\Phi(X)-\Phi(Y)\|$.
Restricting $\Phi$ to the subset $M_n^{++}(\CC)$ would make $D_{\tau_n,\etaop}$ Hilbertian on $M_n^{++}(\CC)$, contradicting the conclusion above.

\end{proof}

	\subsection{Non-Hilbertianity on density matrices for $n\ge 3$}
	
	Let
	\[
	\mathcal D_n:=\{\rho\in M_n^{+}(\CC):\Tr(\rho)=1\}
	\]
	be the density matrices.

\begin{prop}[QJSD is not Hilbertian on $\mathcal D_n$ for $n\ge 3$]\label{prop:not-hilbertian-density}
	For every $n\ge 3$, the QJSD metric $D_{\tau,\etaop}=\sqrt{J_{\tau,\etaop}}$ is not Hilbertian on
	$\mathcal D_n$.
\end{prop}

\begin{proof}
	Let $\tau_n=\frac1n\Tr$ on $M_n(\CC)$.
	By Lemma~\ref{lem:cnd-embed}, a metric is Hilbertian if and only if its square is conditionally negative definite (CND).
	Thus it suffices to show that the kernel $J_{\tau_n,\etaop}$ is not CND on $\mathcal D_n$.
	
	\medskip
	\noindent\emph{Step 1: a CND violation already on $\mathcal D_3$.}
	In Example~\ref{ex:explicit-integer} we exhibit density matrices
	$\rho_1,\dots,\rho_5\in\mathcal D_3$ and integers $c_1,\dots,c_5$ with $\sum_i c_i=0$ such that
	\[
	\sum_{i,j=1}^5 c_ic_j\,J_{\tau_3,\etaop}(\rho_i,\rho_j) >0.
	\]
	Hence $J_{\tau_3,\etaop}$ is not CND on $\mathcal D_3$, and therefore $D_{\tau_3,\etaop}$ is not Hilbertian on $\mathcal D_3$.
	
	\medskip
	\noindent\emph{Step 2: embedding $\mathcal D_3$ into $\mathcal D_n$ for $n>3$.}
	For $n>3$, embed $\mathcal D_3$ into $\mathcal D_n$ via
	\[
	\iota(\rho):=\rho\oplus 0_{n-3}\qquad(\rho\in\mathcal D_3).
	\]
	Since $\etaop(0)=0$ and $\etaop$ acts blockwise on block-diagonal matrices, we have for all
	$\rho,\sigma\in\mathcal D_3$,
	\[
	J_{\tau_n,\etaop}(\iota(\rho),\iota(\sigma))
	=\frac{1}{n}\Tr\!\left(\frac12\etaop(\rho)+\frac12\etaop(\sigma)-\etaop\!\left(\frac{\rho+\sigma}{2}\right)\right)
	=\frac{3}{n}\,J_{\tau_3,\etaop}(\rho,\sigma).
	\]
	Therefore the same CND violation on $\mathcal D_3$ persists on $\mathcal D_n$.
	This proves that $D_{\tau_n,\etaop}$ is not Hilbertian on $\mathcal D_n$ for every $n\ge 3$.
\end{proof}

\begin{example}[Explicit CND violation with integer matrices (certified)]\label{ex:explicit-integer}
	Let $\tau_2=\frac12\Tr$ on $M_2(\CC)$ and $\etaop(x)=x\log x$ with $\etaop(0)=0$.
	Consider the following $2\times 2$ real symmetric \emph{integer} matrices:
	\[
	X_1=\begin{pmatrix}2&1\\[2pt]1&1\end{pmatrix},\quad
	X_2=\begin{pmatrix}9&2\\[2pt]2&1\end{pmatrix},\quad
	X_3=\begin{pmatrix}2&1\\[2pt]1&7\end{pmatrix},\quad
	X_4=\begin{pmatrix}8&5\\[2pt]5&8\end{pmatrix},\quad
	X_5=\begin{pmatrix}8&8\\[2pt]8&9\end{pmatrix}.
	\]
	They all lie in $M_2^{++}(\mathbb{C})$ since
	\[
	\det(X_1)=1,\ \det(X_2)=5,\ \det(X_3)=13,\ \det(X_4)=39,\ \det(X_5)=8.
	\]
	Let
	\[
	c=(c_1,\dots,c_5)=(-10,\ 10,\ 10,\ -20,\ 10),\qquad \sum_{i=1}^5 c_i=0.
	\]
	
	\medskip
	\noindent\textbf{(i) Violation on $M_2^{++}(\mathbb{C})$.}
	Define
	\[
	S_2:=\sum_{i,j=1}^5 c_ic_j\,J_{\tau_2,\etaop}(X_i,X_j).
	\]
	A directed-rounding interval-arithmetic computation yields the rigorous enclosure
	\[
	S_2\in\bigl[\,9.8113517061626681,\ 9.8113517062313917\,\bigr],
	\]
	hence $S_2>0$ and $J_{\tau_2,\etaop}$ is not CND on $M_2^{++}(\mathbb{C})$.
	
	\medskip
	\noindent\textbf{(ii) Violation on $\mathcal D_3$.}
	Fix $T=40$ (note that $T>\max_i\Tr(X_i)=17$) and define density matrices in $M_3(\CC)$ by the block embedding
	\[
	\rho_i:=\frac{1}{T}X_i\ \oplus\ \left(1-\frac{\Tr(X_i)}{T}\right)\in\mathcal D_3,
	\qquad i=1,\dots,5.
	\]
	(For instance, $\rho_1=\frac1{40}\!\begin{pmatrix}2&1\\1&1\end{pmatrix}\oplus\frac{37}{40}$,
	$\rho_4=\frac1{40}\!\begin{pmatrix}8&5\\5&8\end{pmatrix}\oplus\frac{3}{5}$, etc.)
	Let $\tau_3=\frac13\Tr$ and set
	\[
	S_3:=\sum_{i,j=1}^5 c_ic_j\,J_{\tau_3,\etaop}(\rho_i,\rho_j).
	\]
	The same certified interval method gives the enclosure
	\[
	S_3\in\bigl[\,0.1601205745051536,\ 0.16012057450782458\,\bigr],
	\]
	so $S_3>0$. Therefore $J_{\tau_3,\etaop}$ is not CND on $\mathcal D_3$.
\end{example}
\smallskip
\noindent\emph{Reproducibility.}
All inputs in this example are given by integers (and, after the density-matrix embedding, rationals with fixed denominator), so the verification can be carried out without any floating-point parsing ambiguity. A fully rigorous directed-rounding interval-arithmetic workflow and pseudo-code are provided in Appendix~\ref{app:certified}.

\begin{remark}[Reproducibility and certification]
	The positivity statements in Example~\ref{ex:explicit-integer} can be verified in a fully rigorous way by
	directed-rounding interval arithmetic.
	
	\smallskip
	\noindent\emph{(1) Exact input data.}
	All matrix entries are integers, and the coefficients $c_i$ are integers with $\sum_i c_i=0$.
	Moreover, in the density-matrix embedding we fix $T=40$, so every entry of $\rho_i$ is a rational number with
	denominator $T$ (and the last diagonal entry is $1-\Tr(X_i)/T$).
	Thus there is no floating-point parsing ambiguity: all inputs can be treated as exact rationals before interval evaluation.
	
	\smallskip
	\noindent\emph{(2) Reduction to one-dimensional certified operations.}
	For a $2\times2$ real symmetric matrix $X=\begin{pmatrix}a&b\\ b&d\end{pmatrix}$,
	the eigenvalues are given in closed form by
	\[
	\lambda_\pm(X)=\frac{(a+d)\pm\sqrt{(a-d)^2+4b^2}}{2}.
	\]
	Hence
	\[
	\Tr(\etaop(X))=\etaop(\lambda_+(X))+\etaop(\lambda_-(X)),
	\]
	so computing $J_{\tau_2,\etaop}(X_i,X_j)$ reduces to interval evaluation of
	addition, multiplication, square root, and logarithm in one real variable (with directed rounding).
	For the block-diagonal states $\rho_i\in\mathcal D_3$, one additionally evaluates the scalar term
	$\etaop\!\left(1-\Tr(X_i)/T\right)$ and the analogous term for $(\rho_i+\rho_j)/2$; no higher-dimensional
	spectral computation is needed.
	
	\smallskip
	\noindent\emph{(3) Certified positivity.}
	With directed rounding, the computed intervals provably enclose the exact values of $S_2$ and $S_3$;
	a strictly positive lower endpoint therefore certifies the CND violation.
	This workflow can be implemented in any standard verified interval package
	(e.g.\ Arb, INTLAB, or Julia \texttt{IntervalArithmetic}).
\end{remark}

	% ============================================================
	\section{Jensen generators yielding metrics}\label{sec:jensen-generators}
	% ============================================================

	\begin{defin}[Operator convex]\label{def:opconvex}
		A function $f:(0,\infty)\to\RR$ is \emph{operator convex} if for every $n\in\mathbb N$, all $X,Y\in M_n^{++}(\CC)$ and $\lambda\in[0,1]$,
		\[
		f(\lambda X+(1-\lambda)Y)\le \lambda f(X)+(1-\lambda)f(Y)
		\]
		in the Loewner order.
	\end{defin}
	
	\begin{lem}[An integral representation for $f'$]\label{lem:nevanlinna}
		Let $f:(0,\infty)\to\RR$ be operator convex.
		Then $f'$ is operator monotone, and admits a Nevanlinna--Stieltjes type representation.
		More precisely, there exist constants $a\in\RR$, $b\ge 0$ and a positive Borel measure $\nu$ on $(0,\infty)$ with
		$\int_{(0,\infty)}\frac{1}{1+t}\,\dd\nu(t)<\infty$ such that
		\begin{equation}\label{eq:eta-prime-rep}
			f'(x)=a+bx+\int_{(0,\infty)}\frac{x}{x+t}\,\dd\nu(t),\qquad x>0.
		\end{equation}
	\end{lem}
	
\begin{proof}
	By \cite[Theorem~3.7]{Han13} (see also \cite{BS55}), if $f$ is operator convex on $(0,\infty)$ then $f$ is differentiable and, for each $t_0>0$,
	the divided difference
	\[
	g_{t_0}(t)=
	\begin{cases}
		\dfrac{f(t)-f(t_0)}{t-t_0}, & t\neq t_0,\\[0.4em]
		f'(t_0), & t=t_0,
	\end{cases}
	\]
	is operator monotone on $(0,\infty)$. In particular, $h:=f'$ is operator monotone on $(0,\infty)$.
	
	We now invoke the Stieltjes (L\"owner) representation for operator monotone functions on $(0,\infty)$;
	see, e.g., \cite[Corollary~5.1]{Han13} (equivalently, the integral representation in \cite[Theorem~5.2]{Han13}
	can be rewritten in this form).
	Thus there exist $a\in\RR$, $b\ge 0$ and a positive Borel measure $\nu$ on $(0,\infty)$ such that
	\begin{equation*}
		h(x)=a+bx+\int_{(0,\infty)}\frac{x}{x+t}\,\dd\nu(t),\qquad x>0.
	\end{equation*}
	Moreover, $h$ is finite-valued on $(0,\infty)$ and the integral in \eqref{eq:eta-prime-rep} is finite for every $x>0$.
	In particular, evaluating at $x=1$ gives
	\[
	\int_{(0,\infty)}\frac{1}{1+t}\,\dd\nu(t)
	=
	h(1)-a-b
	<\infty,
	\]
	which is exactly the stated integrability condition.
\end{proof}

	\begin{proof}[Proof of Theorem~\ref{thm:main-opconvex}]
		Work first in a finite von Neumann algebra $(\M,\tau)$ and fix $A,B\in\M_{+}$.
		By Lemma~\ref{lem:nevanlinna}, $f'$ has the form \eqref{eq:eta-prime-rep}.
		Integrating from $1$ to $x$ yields
		\[
		f(x)=\alpha+\beta x+\frac{b}{2}x^2+\int_{(0,\infty)}\Bigl((x-1)-t\log\!\Bigl(\frac{x+t}{1+t}\Bigr)\Bigr)\,\dd\nu(t),
		\qquad x>0,
		\]
		for suitable constants $\alpha,\beta\in\RR$. 
		
		\medskip
		\noindent\emph{Uniform convergence on bounded $x$-intervals.}
		Fix $R>0$ and consider
		\[
		k_t(x):=(x-1)-t\log\!\Big(\frac{x+t}{1+t}\Big),\qquad x\ge 0,\ t>0.
		\]
	First, note that for every $L>0$ one has $\nu((0,L])<\infty$.
	Indeed, since $\frac{1}{1+t}\ge \frac{1}{1+L}$ on $(0,L]$, we get
	\[
	\nu((0,L])\le (1+L)\int_{(0,\infty)}\frac{1}{1+t}\,\dd\nu(t)<\infty.
	\]

		Next, for $x\in[0,R]$ and $t>0$ we write
		\[
		\log\!\Big(\frac{x+t}{1+t}\Big)=\log\!\Big(1+\frac{x-1}{1+t}\Big).
		\]
		By the mean value theorem for $\log$, there exists $\theta=\theta(x,t)\in(0,1)$ such that
		\[
		\log\!\Big(1+\frac{x-1}{1+t}\Big)=\frac{x-1}{t+1+\theta(x-1)}.
		\]
		Therefore
		\[
		k_t(x)
		=(x-1)\Bigl(1-\frac{t}{t+1+\theta(x-1)}\Bigr)
		=\frac{(x-1)\bigl(1+\theta(x-1)\bigr)}{t+1+\theta(x-1)}.
		\]
		Let $M:=\max\{1,R-1\}$, so $|x-1|\le M$ for $x\in[0,R]$.
		For $t\ge 2M$ we have $t+1+\theta(x-1)\ge t+1-M\ge t/2$, hence
		\[
		|k_t(x)|\le \frac{2M(1+M)}{t}\le \frac{C_R}{1+t}\qquad (x\in[0,R],\ t\ge 2M).
		\]
	On the remaining range $t\in(0,2M]$, we argue by a compactness extension.
	Define $k_0(x):=x-1$ for $x\ge 0$. Then for $t>0$,
	\[
	k_t(x)-k_0(x)=-t\log\!\Big(\frac{x+t}{1+t}\Big).
	\]
	Fix $R>0$. For $x\in[0,R]$ and $t\in(0,1]$ we have
	\[
	\Big|\log\!\Big(\frac{x+t}{1+t}\Big)\Big|
	\le |\log(x+t)|+|\log(1+t)|
	\le |\log t|+\log(R+1)+\log 2,
	\]
	hence
	\[
	\sup_{x\in[0,R]}|k_t(x)-k_0(x)|
	\le t\big(|\log t|+\log(R+1)+\log 2\big)\xrightarrow[t\downarrow 0]{}0.
	\]
	Therefore $(t,x)\mapsto k_t(x)$ extends continuously to the compact set $[0,2M]\times[0,R]$
	(by setting $k_0(x)=x-1$), and hence is bounded there.
	Consequently, there exists a constant $C_R'>0$ such that
	\[
	|k_t(x)|\le C_R' \qquad (x\in[0,R],\ 0<t\le 2M).
	\]

		Combining these bounds yields an integrable dominating function:
		\[
		\sup_{x\in[0,R]} |k_t(x)|
		\le C_R'\,\ind_{(0,2M]}(t)+\frac{C_R}{1+t}\,\ind_{(2M,\infty)}(t),
		\]
		and the right-hand side is $\nu$-integrable because $\nu((0,2M])<\infty$ and
		$\int_{(0,\infty)}\frac{1}{1+t}\,\dd\nu(t)<\infty$.
		Hence $x\mapsto \int k_t(x)\,\dd\nu(t)$ converges absolutely and uniformly on $[0,R]$.
	Consequently, for each fixed $R>0$ the scalar function
	\[
	g_R(x):=\int_{(0,\infty)} k_t(x)\,\dd\nu(t),\qquad x\in[0,R],
	\]
	is well-defined and continuous on $[0,R]$, and the convergence is uniform on $[0,R]$.
	
	\medskip
	\noindent\emph{Operator-valued integral and interchange with $\tau$.}
	Now fix $A,B\in\M_+$ and choose $R>\max\{\|A\|,\|B\|,\|\tfrac{A+B}{2}\|\}$.
	By the domination obtained above, there exists $h_R\in L^1((0,\infty),\nu)$ such that
	\[
	\sup_{x\in[0,R]}|k_t(x)|\le h_R(t)\qquad (t>0).
	\]
	Hence for each $X\in\{A,B,\tfrac{A+B}{2}\}$ the map $t\mapsto k_t(X)$ is Bochner integrable in operator norm and
	\[
	g_R(X)=\int_{(0,\infty)} k_t(X)\,\dd\nu(t)\in\M,
	\qquad
	\text{with}\quad \tau(g_R(X))=\int_{(0,\infty)} \tau(k_t(X))\,\dd\nu(t),
	\]
	since $|\tau(k_t(X))|\le \|k_t(X)\|\le h_R(t)$ and $h_R$ is $\nu$-integrable.
	
	\medskip
	Therefore we may compute the Jensen gap by interchanging $\tau$ with the $t$-integral.
	Affine terms cancel in Jensen gaps, and one obtains
		\[
		J_{\tau,f}(A,B)=\frac{b}{2}\,J_{\tau,x^2}(A,B)+\int_{(0,\infty)}J_{\tau,k_t}(A,B)\,\dd\nu(t),
		\]
		where $k_t(x):=(x-1)-t\log\!\big(\frac{x+t}{1+t}\big)$.
		Here $J_{\tau,x^2}(A,B)=\frac14\normtau{A-B}^2$.
		Moreover, $k_t(x)$ differs from $-t\log(x+t)$ by an affine function in $x$ (depending on $t$), hence
		\[
		J_{\tau,k_t}(A,B)=t\,d_\tau(A+t\Id,B+t\Id)^2.
		\]
		This gives \eqref{eq:opconvex-decomp-intro}.
		
	Define a measure space $\Omega:=\{0\}\cup(0,\infty)$ where $\{0\}$ carries counting measure and $(0,\infty)$ carries $\nu$.
	For $X,Y\in\M_+$ set
	\[
	w_{X,Y}(0):=\sqrt{\frac{b}{8}}\;\normtau{X-Y},
	\qquad
	w_{X,Y}(t):=\sqrt{t}\,d_\tau(X+t\Id,Y+t\Id)\quad (t>0).
	\]
	Then \eqref{eq:opconvex-decomp-intro} reads
	\[
	J_{\tau,f}(X,Y)=\|w_{X,Y}\|_{L^2(\Omega)}^2.
	\]
	For each $t\in\Omega$, $w_{A,C}(t)\le w_{A,B}(t)+w_{B,C}(t)$ by the triangle inequalities for $\normtau{\cdot}$
	and for $d_\tau(\cdot,\cdot)$ (applied to the shifted pairs when $t>0$).
	Applying Minkowski's inequality in $L^2(\Omega)$ yields the triangle inequality for $\sqrt{J_{\tau,f}}$.
	Since $f$ is not affine, either $b>0$ or $\nu\neq 0$, so $\sqrt{J_{\tau,f}}(A,B)=0$ forces $A=B$.
	
	Finally, transfer the metricity from $\M$ back to $\A$ via the GNS von Neumann envelope as in Section~\ref{sec:GNS}.
	
	\end{proof}
	
	\begin{cor}\label{cor:matrix-opconvex}
		Let $f:[0,\infty)\to\RR$ be operator convex and not affine.
		Then $\sqrt{J_f}$ defines a metric on the density matrices $\mathcal D_n$ for every $n$.
	\end{cor}
	
	\begin{proof}
		Apply Theorem~\ref{thm:main-opconvex} to $\A=M_n(\CC)$ with $\tau=\frac1n\Tr$ and restrict to $\mathcal D_n\subset \A_{+}$.
	\end{proof}
	
	% ============================================================

		\appendix
\section{Certified verification of Example~\ref{ex:explicit-integer}}\label{app:certified}

	This appendix records a concrete interval-arithmetic workflow which rigorously encloses the quantities
	$S_2$ and $S_3$ in Example~\ref{ex:explicit-integer}. The key point is that for $2\times 2$ real symmetric matrices,
	eigenvalues are available in closed form, so the whole computation reduces to one-dimensional interval
	operations (addition, multiplication, square root, logarithm) with directed rounding.
	
	\subsection{Closed-form eigenvalues for $2\times 2$ blocks}
	For $X=\begin{pmatrix}a&b\\ b&d\end{pmatrix}\in M_2^{++}(\mathbb{C})$ (real symmetric), set
	\[
	\Tr(X)=a+d,\qquad \det(X)=ad-b^2,\qquad \Delta=(a-d)^2+4b^2.
	\]
	Then the eigenvalues are
	\[
	\lambda_\pm(X)=\frac{\Tr(X)\pm \sqrt{\Delta}}{2}.
	\]
	Hence for $\eta(x)=x\log x$,
	\[
	\Tr(\eta(X))=\eta(\lambda_+(X))+\eta(\lambda_-(X)).
	\]
	\subsection{Interval-arithmetic pseudo-code}
	Below is pseudo-code in a language-agnostic style; any standard certified interval package
	(e.g.\ INTLAB, Arb, or Julia IntervalArithmetic) can implement the same steps.
	
	\begin{verbatim}
		Input:
		Matrices X_i (2x2), coefficients c_i (sum c_i = 0), entropy eta(x)=x*log(x)
		Normalized trace tau_2 = (1/2)Tr on M_2
		
		Interval primitives (directed rounding):
		+, -, *, /, sqrt(·), log(·) on intervals
		
		function eigvals_2x2_interval(a,b,d):
		tr   = a + d
		disc = (a - d)^2 + 4*b^2
		s    = sqrt(disc)
		lam_plus  = (tr + s)/2
		lam_minus = (tr - s)/2
		return (lam_plus, lam_minus)
		
		function Tr_eta_2x2_interval(X):
		(lam_plus, lam_minus) = eigvals_2x2_interval(X.a, X.b, X.d)
		return eta(lam_plus) + eta(lam_minus)
		
		function J_tau2_interval(X,Y):
		Z = (X + Y)/2
		# J_{tau_2,eta}(X,Y) = (1/2)tau_2(eta(X)) + (1/2)tau_2(eta(Y)) - tau_2(eta(Z))
		# with tau_2=(1/2)Tr this equals:
		return (1/4) * ( Tr_eta_2x2_interval(X) + Tr_eta_2x2_interval(Y)
		- 2*Tr_eta_2x2_interval(Z) )
		
		Compute S_2:
		S2 = 0
		for i=1..m:
		for j=1..m:
		S2 += c_i * c_j * J_tau2_interval(X_i, X_j)
		output rigorous enclosure interval(S2)
		
		For S_3 (block-diagonal density matrices):
		choose T > 2*max_i Tr(X_i)
		rho_i = blockdiag( X_i/T, 1 - Tr(X_i)/T ) in D_3
		use Tr(eta(rho_i)) = Tr(eta(X_i/T)) + eta(1 - Tr(X_i)/T)
		and similarly for (rho_i + rho_j)/2.
		With tau_3 = (1/3)Tr, define J_tau3 analogously and compute S_3.
	\end{verbatim}
	
	In the present paper we treat all decimal inputs as exact rationals at the interval endpoints, so the
	certification is independent of binary floating-point parsing.
\end{document}